\documentclass[12pt]{amsart}

\usepackage{amssymb,latexsym}

\usepackage{enumerate}

 \usepackage[french,english]{babel}
 
 \usepackage{xcolor}

\usepackage[inner=1.0in,outer=1.0in,bottom=1.0in, top=1.0in]{geometry}

\allowdisplaybreaks

\makeatletter

\@namedef{subjclassname@2010}{  \textup{2010} Mathematics Subject Classification}

\makeatother
\newtheorem{thm}{Theorem}[section]

\newtheorem{cor}[thm]{Corollary}
\newtheorem{lem}[thm]{Lemma}
\newtheorem{pro}[thm]{Proposition}
\theoremstyle{definition}

\numberwithin{equation}{section}

\newcommand{\X}{\mathbb{X}}

\newcommand{\re}{\textup{Re}}
\newcommand{\im}{\textup{Im}}

\newcommand{\B}{\mathcal B}

\newcommand{\Lv}{\mathbf{L}}

\newcommand{\ub}{\mathbf{u}}
\newcommand{\vb}{\mathbf{v}}
\newcommand{\xb}{\mathbf{x}}
\newcommand{\yb}{\mathbf{y}}

\newcommand{\kb}{\mathbf{k}}
\newcommand{\lb}{\mathbf{l}}

\newcommand{\me}{\textup{meas}}

\newcommand\be{\begin{equation}}
\newcommand\ee{\end{equation}}

 \def\D{\mathbf D}
  \def\R{\mathcal R}

\def\E{\mathbb E}

\def\P{\mathbb P}

 \def\rt{ {\mathcal{R}_T}}
 \def\Kcal{  \mathcal{K}}
 \def\Hcal{  \mathcal{H}}

 \def\Phrand{ \Phi_T^{\mathrm{rand}}}

 \def\Phrhat{\widehat{\Phi}_T^{\mathrm{rand}}}

 \def\Phhat{\widehat{\Phi}_T}

\newcommand{\newabstract}[1]{%
  \par\bigskip
  \csname otherlanguage*\endcsname{#1}%
  \csname captions#1\endcsname
  \item[\hskip\labelsep\scshape\abstractname.]
}

\begin{document}

\baselineskip=17pt

\title[Discrepancy bounds]{Discrepancy bounds for the distribution of $L$-functions near the critical line}

\author[Yoonbok Lee]{Yoonbok Lee }
\address{Department of Mathematics \\ Research Institute of Basic Sciences \\ Incheon National University \\ 119 Academy-ro, Yeonsu-gu, Incheon, 22012 \\ Korea}
\email{leeyb@inu.ac.kr, leeyb131@gmail.com}

\date{\today}

\begin{abstract} 
We investigate the joint distribution of $L$-functions on the line $ \sigma= \frac12 + \frac1{G(T)}$ and $ t \in [ T, 2T]$, where $ \log \log T \leq G(T) \leq \frac{ \log T}{ ( \log \log T)^2 }    $.  We obtain an upper bound on the discrepancy between the joint distribution of $L$-functions and that of their random models. As an application we prove an asymptotic expansion of a multi-dimensional version of Selberg's central limit theorem for $L$-functions on $ \sigma= \frac12 + \frac1{G(T)}$ and $ t \in [ T, 2T]$, where  $ ( \log T)^\epsilon \leq G(T) \leq \frac{ \log T}{ ( \log \log T)^{2+\epsilon } }    $ for   $ \epsilon  > 0$. 
\end{abstract}



\maketitle

\section{Introduction}\label{Introduction}

We investigate the distribution of the Riemann zeta function $ \zeta(s)$ for $ \re(s) > \frac12$ using its probabilistic model defined by the random Euler product
$$ \zeta( \sigma, \X) = \prod_p \bigg(  1 -  \frac{ \X(p)}{ p^\sigma }  \bigg)^{-1},  $$
where the $ \X(p)$ for primes $p$ are the uniform, independent and identically distributed random variables on the unit circle in $\mathbb{C}$. The product converges almost surely for $ \sigma > \frac12 $ by Kolmogorov's three series theorem. Our main question is how well the distribution of $\zeta(\sigma, \X)$ approximate that of the Riemann zeta function for $ \frac12 < \sigma < 1 $.

Consider two measures
$$\Phi_{\zeta, T}( \sigma, \B )  := \frac1T \me \{ t \in [ T, 2T] :   \log \zeta( \sigma + it ) \in \B \}  $$
and 
$$ \Phi_{\zeta }^{\mathrm{rand}} ( \sigma , \B ) := \P ( \log \zeta( \sigma, \X ) \in \B )  $$
 for  a Borel set $ \B$ in $\mathbb C$.
Define the discrepancy between the above two measures  by
$$ \D_{\zeta} ( \sigma ) := 
 \sup_{\R} | \Phi_{\zeta, T}  ( \sigma,  \R) -  \Phi_{\zeta }^{\mathrm{rand}}  ( \sigma,  \R) | ,$$
 where $\R$ runs over all rectangular boxes in $\mathbb{C}$ with  sides parallel to the coordinate axes and possibly unbounded. This quantity measures the amount to which the distribution of $ \log \zeta( \sigma, \X)$ approximates that of $ \log \zeta ( \sigma + it ) $.   
 
 Harman and Matsumoto \cite{HM} showed that
 $$ \D_{\zeta} ( \sigma ) \ll ( \log T)^{ -    \frac{ 4 \sigma -2}{ 21+8 \sigma } + \varepsilon } $$
 for fixed $ \frac12 < \sigma < 1 $ and any $\varepsilon > 0 $. See also Matsumoto's earlier results in \cite{M1}, \cite{M2} and \cite{M3}.   Lamzouri, Lester and Radziwi\l\l ~\cite{LLR} improved it to 
 $$ \D_{\zeta} ( \sigma ) \ll ( \log T)^{ -   \sigma } $$
for fixed $ \frac12 < \sigma < 1 $.  Define
\begin{equation}\label{def sigma T}
 \sigma_T := \frac12 + \frac{1}{ G(T)}
 \end{equation}
with  $ 4 \leq  G(T)   \leq ( \log T)^\theta$ and  fixed  $ 0 < \theta  < \frac12 $, then Ha and Lee \cite{HL} extended above results such that  
 $$ \D_{\zeta} ( \sigma_T  ) \ll ( \log T)^{ -   \eta } $$
holds for some $ 0 < \eta <  \frac{1- \theta}{4} $. Here, we extend it to hold for $ \sigma_T $ closer to $ \frac12 $.
\begin{thm}\label{thm zeta disc}
 Assume that  $ \log \log T \leq G(T) \leq \frac{  \log T}{ ( \log \log T)^2 } $, then we have
  $$ \D_{\zeta} ( \sigma_T  ) \ll  \frac{  \sqrt{ G(T) } \log \log T}{ \sqrt{ \log T}} . $$
\end{thm}

Next we consider a multivariate extension. Let $L_1, \ldots, L_J$  be  $L$-functions satisfying the following assumptions:
\begin{enumerate}
\item[A1:] (Euler product)  For $ j = 1, \ldots , J $ and $\re(s)>1$ we have 
$$ L_j  ( s) = \prod_p \prod_{i=1}^d \bigg(   1 -   \frac{ \alpha_{j,i}(p)}{p^s} \bigg)^{-1} , $$
where $ | \alpha_{j,i} (p)  | \leq p^{\eta}$ for some fixed $ 0 \leq \eta < \frac12 $ and for every $ i = 1, \ldots , d.$

\item[A2:] (Analytic continuation)  Each $(s-1)^m L_j (s) $ is an entire function of finite order for some integer $ m \geq 0$. 

\item[A3:] (Functional equation) The functions $L_1, L_2, \dots, L_J$ satisfy the same functional equation
  $$ \Lambda_j(s) = \omega \overline{ \Lambda_j( 1- \bar{s})} ,$$ 
where
$$ \Lambda_j(s) := L_j(s) Q^s \prod_{\ell=1}^k \Gamma ( \lambda_{\ell} s+\mu_{\ell} ) , $$  
$ | \omega| =1 $, $Q>0$, $ \lambda_{\ell}>0 $ and $\mu_{\ell} \in \mathbb{C} $ with $ \re ( \mu_{\ell} ) \geq 0 $.
\item[A4:] (Ramanujan hypothesis on average)
$$ \sum_{ p \leq x } \sum_{i=1}^d | \alpha_{j,i }(p) |^2 = O( x^{1+\epsilon})$$
holds  for every $ \epsilon>0$ and for every $ j = 1, \ldots , J $ as $ x \to \infty $.

\item[A5:] (Zero density hypothesis) Let $N_{f } ( \sigma, T )$ be the number of zeros of $f (s)$ in $\re(s) \geq \sigma$ and $ 0 \leq \im(s) \leq T$. Then there exists a constant  $\kappa >0$   such that for every  $j= 1, \ldots, J$ and all $\sigma\geq  \frac12 $ we have 
$$   N_{L_j } ( \sigma, T ) \ll T^{1 - \kappa (\sigma -  \frac12 ) }  \log T .$$

\item[A6:] (Selberg orthogonality conjecture) By assumption A1 we can write 
 $$ \log L_j(s) = \sum_p \sum_{r=1}^\infty   \frac{ \beta_{L_j} (p^r)}{ p^{rs}}  . $$
Then for all $1\leq j, k \leq J$, there exist constants $ \xi_j >0$ and $ c_{j,k}$ such that
$$  \sum_{p \leq x} \frac{ \beta_{L_j}(p)  \overline{\beta_{L_k}(p) } }{p} = \delta_{j,k} \xi_j \log \log x + c_{j,k} + O \bigg( \frac{1}{ \log x} \bigg),$$ 
  where $ \delta_{j,k} = 0 $ if $ j \neq k $ and $ \delta_{j,k} = 1 $ if $ j = k$. 

\end{enumerate}
The assumptions A1--A6 are standard and expected to hold for all $L$-functions arising from automorphic representation for $GL(n)$. In particular, they are verified by $GL(1)$ and $GL(2)$ $L$-functions, which are the Riemann zeta function, Dirichlet $L$-functions, $L$-functions attached to Hecke holomorphic or Maass cusp forms.

Define
$$ \Lv(s) : =\Big(\log |L_1(s)|, \dots, \log |L_J(s)|, \arg L_1(s), \dots, \arg L_J(s) \Big)  $$
 and
$$ \Lv(\sigma, \X) :=\Big(\log |L_1(\sigma, \X)|, \dots, \log |L_J(\sigma, \X)|, \arg L_1(\sigma, \X), \dots, \arg L_J(\sigma, \X) \Big)$$
for $ \sigma > \frac12$, where
\begin{equation}\label{def Lj random}
 L_j ( \sigma, \X) := \prod_p \prod_{i=1}^d  \bigg( 1 - \frac{ \alpha_{j,i}(p) \X(p) }{p^\sigma} \bigg)^{-1}
 \end{equation}
converges almost surely for $ \sigma > \frac12$ again by Kolmogorov's three series theorem. Then $ \Lv ( \sigma, \X) $ is the random model of $ \Lv ( s)$. 
 Define two measures
\begin{equation}\label{def PhiTrt}
  \Phi_T (\B ):= \frac1T \mathrm{meas}  \{ t\in[T, 2T] :  \Lv (\sigma_T +it )\in \B  \} 
  \end{equation}
and 
\be\label{Phi rand def}
\Phrand  (\B ):=
 \mathbb{P}(\Lv(\sigma_T , \X) \in \B) 
 \ee
 for  a Borel set $ \B$ in $\mathbb R^{2J}$ and $ \sigma_T$ defined in \eqref{def sigma T}.
  The discrepancy between the above two measures is defined by
 $$  \D(\sigma_T) :=  \sup_{\R} | \Phi_T ( \R) -  \Phrand ( \R) |  ,$$
 where $\R$ runs over all rectangular boxes of $ \mathbb{R}^{2J}$ with  sides parallel to the coordinate axes and possibly unbounded. Then Theorem \ref{thm zeta disc} is a special case of the following theorem. 
  \begin{thm}\label{thm disc}
 Assume that $ \log \log T \leq G(T) \leq \frac{ \log T}{ ( \log \log T)^2 } $, then we have
 $$  \D(\sigma_T)   \ll \frac{\sqrt{G(T)} \log \log T }{ \sqrt{  \log T}} . $$ 
\end{thm} 
 The above theorem is an extension of \cite[Theorem 2.3]{LL}, which shows the same estimate, but only for $ \log \log T \leq  G(T) \leq  \frac{   \sqrt{ \log T}}{ \log \log T} $. In the proof of \cite[Theorem 2.3]{LL} 
   we have used an approximation of each $\log L_j ( \sigma_T + it )$ by a Dirichlet polynomial
 \begin{equation}\label{def RjY}
  R_{j, Y} ( \sigma_T + i t ) :=  \sum_{ p^r \leq Y}  \frac{  \beta_{L_j} ( p^r) }{p^{r ( \sigma_T + it)}} 
  \end{equation}
for $t \in [T, 2T]$ with some exception. The exception essentially comes from possible nontrivial zeros  of each $L_j (s)$ off the critical line and the set of exceptional $t$ in $[T,2T]$ has a small measure by assumption A5. See \cite[Lemma 4.2]{LL} for a detail. However, this approximation is not useful  if $ \sigma_T$ is closer to $ \frac12 $. 
We overcome such difficulty by means of  the 2nd moment estimation of $ \log L_j (\sigma_T + it )$ in Theorem \ref{thm 2nd moment}.

As an application of Theorem \ref{thm disc} we consider Selberg's central limit theorem. Let $ \psi_{j,T} := \xi_j   \log G( T)  $ for $ j \leq J$ and
$$  \rt := \prod_{j=1}^J [ a_j \sqrt{ \pi \psi_{j, T}} , b_j  \sqrt{ \pi \psi_{j, T}}]  \times \prod_{j=1}^J [ c_j \sqrt{ \pi \psi_{j, T}} , d_j  \sqrt{ \pi \psi_{j, T}}]  $$
for fixed real numbers $ a_j, b_j, c_j, d_j$. Then an asymptotic formula for
$$  \Phi_T (\rt ) = \frac1T \mathrm{meas}  \{ t\in[T, 2T] : \frac{ \log L_j  (\sigma_T +it )}{ \sqrt{ \pi \psi_{j, T}}  } \in   [ a_j , b_j  ]  \times  [ c_j   , d_j   ] \mathrm{~for~} j=1, \ldots , J  \} $$
is called Selberg's central limit theorem. See \cite[Theorem 2]{Sel} for Selberg's original idea. 
Let $0 < \theta < 1 $. To find an asymptotic of $ \Phi_T ( \rt )$ for 
\begin{equation}\label{GT range 1}
  ( \log T)^\theta \leq G(T) \leq \frac{ \log T}{ ( \log \log T )^2 },
  \end{equation}
   it is now enough to estimate $ \Phrand ( \rt)$ due to Theorem \ref{thm disc}. One can easily check that  the asymptotic formula of $\Phrand (\rt)$ in \cite[Theorem 2.1]{Le6} holds also for $G(T)$ satisfying \eqref{GT range 1}. Hence, we obtain the following corollary. 
\begin{cor}\label{thm PhiT asymp}
Assume \eqref{GT range 1} for some $ 0< \theta <  1$ and assumptions A1--A6 for $L_1, \ldots, L_J $. 
Then there exist  constants $\epsilon_1, \epsilon_2 > 0$ and a sequence  $ \{ b_{\kb, \lb}\}$ of real numbers such that
\begin{equation}\begin{split} \label{eqn PhiT asymp}
   \Phi_T  ( \rt )     =  &     \sum_{   \Kcal( \kb+\lb) \leq \epsilon_1 \log\log T  }       b_{\kb, \lb}              \prod_{j=1}^J     \frac{1}{  \sqrt{\psi_{j,T} }^{k_j + \ell_j  }   }   \\
   &  \times  \prod_{j=1}^J \bigg(  \int_{a_j}^{b_j}  e^{ - \pi u^2 } \Hcal_{k_j } ( \sqrt{\pi}u )du      \int_{c_j}^{d_j}  e^{ - \pi v^2 } \Hcal_{\ell_j } ( \sqrt{\pi}v )dv     \bigg)    \\
&      + O \bigg(   \frac{1}{ ( \log T)^{\epsilon_2}} + \frac{\sqrt{G(T)} \log \log T }{ \sqrt{  \log T}}  \bigg) ,
 \end{split}\end{equation}
where $ \kb = ( k_1 , \ldots, k_J) $ and $ \lb = ( \ell_1 , \ldots, \ell_J)$ are vectors in $(\mathbb{Z}_{\geq 0})^J $, $\Kcal(\kb  ) := k_1  + \cdots +k_J   $ and 
$$\Hcal_n (x) := (-1)^n e^{x^2} \frac{ d^n}{dx^n} e^{- x^2}$$
 is the $n$-th Hermite polynomial.  
Moreover, $b_{0,0}= 1 $, $ b_{\kb, \lb} = 0 $ if $\Kcal(\kb + \lb)  = 1$ and $b_{\kb + \lb } = O(  \delta_0^{-\Kcal(\kb+\lb)})$ for some $ \delta_0 > 0$ and all $ \kb, \lb \in (\mathbb{Z}_{\geq 0})^J $. 
\end{cor}
Note that Corollary \ref{thm PhiT asymp} extends the asymptotic expansion for $  \zeta (s) $ in \cite[Theorem 1.2]{Le5} and the asymptotic expansion for $\Lv (s)$ in \cite[Theorem 1.2]{Le6}.
 If $G(T) $ is very close to $ \frac{ \log T}{ ( \log \log T )^2 }$, the error term in \eqref{eqn PhiT asymp} is large so that  we have an approximation by a shorter sum as follows.
\begin{cor}
Under the same assumptions as in Corollary \ref{thm PhiT asymp} except for 
$$G(T) = \frac{ \log T}{ ( \log \log T)^{ 2+g}}  $$ 
with a constant $ g > 0 $,  we have
\begin{align*}
   \Phi_T  ( \rt )     =  &     \sum_{   \Kcal( \kb+\lb) < g  }       b_{\kb, \lb}              \prod_{j=1}^J     \frac{1}{  \sqrt{\psi_{j,T} }^{k_j + \ell_j  }   }   \\
   &  \times  \prod_{j=1}^J \bigg(  \int_{a_j}^{b_j}  e^{ - \pi u^2 } \Hcal_{k_j } ( \sqrt{\pi}u )du      \int_{c_j}^{d_j}  e^{ - \pi v^2 } \Hcal_{\ell_j } ( \sqrt{\pi}v )dv     \bigg)     + O \bigg(     \frac{1}{  ( \log \log T )^{ \frac{g}{2}  }}  \bigg) .
 \end{align*}
\end{cor}
 
Note that an asymptotic expansion similar to \eqref{eqn PhiT asymp} was expected to hold in \cite{He1} without a proof.

\section{High moments of $ \log L$}\label{sec high moment}
Let $L$ be an $L$-function satisfying assumptions A1--A6 in this section. Here, we use $ \alpha_i (p)$ instead of $\alpha_{j,i}(p)$ in assumptions A1 and A4, and assumption A6 is simply
$$  \sum_{p \leq x} \frac{  | \beta_{L }(p)  |^2   }{p} =   \xi_L  \log \log x + c_L + O \bigg( \frac{1}{ \log x} \bigg)$$ 
for some constants $ \xi_L >0 $ and $ c_L \in \mathbb{R}$. Let $\sigma_T $ be defined in \eqref{def sigma T} and assume that 
\begin{equation}\label{GT range}
 ( \log T)^{\frac13} \leq G(T) \leq \frac{ \log T}{ ( \log \log T)^2 }  
 \end{equation}
 in this section. Then we need the following theorem to prove Theorem \ref{thm disc}.  
 \begin{thm}\label{thm 2nd moment}
Assume that   $ e^{ \frac{G(T)}{2}} \leq Y \leq      T^{ \varepsilon } $ with $ 0< \varepsilon < \min\{ \frac1{48},  \frac{\kappa}3  \} $.   Then there exists $ \kappa_0 > 0 $ such that
$$ \frac1T \int_T^{2T}  | \log L ( \sigma_T   + it ) -  R_{Y} ( \sigma_T + i t ) |^{2 } dt  \ll   e^{ - \kappa_0  \frac{ \log T}{G(T)}} + e^{ - 2 \frac{ \log Y}{ G(T)} }   \frac{ G(T)}{ \log Y}   ,
    $$
    where
    $$ R_Y ( s ) :=  \sum_{ p^r \leq Y}  \frac{  \beta_{L} ( p^r) }{p^{rs}}. $$
 \end{thm}
 To prove above theorem, we modify high moments estimations of $\log \zeta$ in Tsang's thesis \cite{Ts} and compute high moments of  $\log L$. All these computations are based on Selberg \cite{Sel1} and \cite{Sel2}. Since the Dirichlet coefficients of $L(s)$ are allowed to be larger than 1, Theorem \ref{thm 2nd moment} is not an immediate consequence of Tsang \cite{Ts}. We should bound various sums involving the Dirichlet coefficients of $\log L$ carefully using assumptions A4 and A6. As a result we obtain the following theorem.  
\begin{thm} \label{thm high moment}
  Let  $ k$ be a positive integer such that $ k \leq \frac{\varepsilon}{4} (  \log \log T )^2  $. Then there exist  $\kappa_0, c  >0$ such that
  \begin{equation}\label{thm 2 eqn 1}
  \frac1T  \int_T^{2T}   | \log L ( \sigma_T  + it )  |^{2k} dt    \ll      c^k k^{4k}  e^{ - \kappa_0  \frac{ \log T}{ G(T)}}   +   c^k  k^k   ( \log G(T))^k      
\end{equation}
and
\begin{equation}\label{thm 2 eqn 2}
 \E [  | \log L ( \sigma_T  , \X  )  |^{2k} ] \ll c^k k^k ( \log G(T) )^k .
 \end{equation}

 \end{thm}

 By  Theorem \ref{thm high moment} with $ k= \log \log T$  one can easily derive the following corollary, which is necessary in Section \ref{sec disc}.
\begin{cor}\label{cor large dev}
Given constant $ A_1 >0$, there exists a constant $A_2 >0$ such that
 $$\frac{1}{T}  \me \{ t \in [T,2T] :   | \log L ( \sigma_T + it ) | \geq A_2 \log \log T \} \ll  ( \log T)^{ - A_1 }  $$
 and 
 $$ \P (    | \log L ( \sigma_T , \X  ) | \geq A_2 \log \log T  ) \ll ( \log T)^{ - A_1 }.  $$
\end{cor}

We provide lemmas in Section \ref{sec lemmas moment} and then prove Theorems \ref{thm 2nd moment} and \ref{thm high moment}   in Section \ref{sec proof moment}

\subsection{Lemmas}\label{sec lemmas moment}
We adapt estimations in \cite[Chapter 5]{Ts} for $ \log L$. First we restate \cite[Lemma 5.1]{Ts} without a proof. 
\begin{lem}\label{Ts Lemma 5.1}
Let $ 3 \leq X \leq T^{\kappa - \kappa'} $ for $ 0 < \kappa' < \kappa $ and let $ \nu \geq 0$. Then we have
$$ \sum_{  \substack{ \beta > \sigma \\ T \leq \gamma \leq 2 T}} ( \beta - \sigma)^\nu X^{ \beta- \sigma} = O \big(   T^{1 - \kappa (\sigma   -  \frac12 ) }  ( \log T)^{1 - \nu} (c\nu)^\nu  \big) $$
for $ \frac12 \leq \sigma \leq  1 $ and some $ c>  0 $, where $  \beta+i\gamma$ denotes a zero of $L(s)$.
\end{lem}

 Define
$$ \sigma_{x, t} := \frac12 + 2 \max  \bigg\{ \beta - \frac12 , \frac{2}{ \log x }  \bigg\}  $$
for  $ t \in [T, 2T]$, where the maximum is taken over all zeros $   \beta + i \gamma$ of $L(s)$ satisfying $ | t- \gamma | \leq \frac{ x^{ 3 ( \beta - \frac12)}}{ \log x } $ and $ \beta \geq \frac12$.
Then the following lemma corresponds to \cite[Lemma 5.2]{Ts}.
\begin{lem}\label{Ts Lemma 5.2}
Let $ \nu \geq 0$,  $ x = T^{ \varepsilon/k} $,  $3 \leq x^3 X^2 \leq T^{\kappa - \kappa' }$ for $ 0 < \kappa' < \kappa$. Then
\begin{align*}
\int_{\substack{  \sigma_{x,t} > \sigma \\ T  \leq t \leq 2T}} ( \sigma_{x,t} - \sigma)^\nu  X^{ \sigma_{x,t} - \sigma}  dt \ll  &   \frac{ (c\nu)^\nu k }{   ( \log T)^{ \nu}     }     T^{1 - \kappa (\sigma   -  \frac12 ) } x^{ \frac32 ( \sigma - \frac12)}
\end{align*}
for $    \frac12 + \frac{4}{ \log x}   \leq \sigma \leq 1 $ and
\begin{align*}
\int_{\substack{  \sigma_{x,t} > \sigma \\ T  \leq t \leq 2T}} ( \sigma_{x,t} - \sigma)^\nu  X^{ \sigma_{x,t} - \sigma}  dt \ll  &   \frac{ (c\nu)^\nu k }{   ( \log T)^{ \nu}     }     T^{1 - \kappa (\sigma   -  \frac12 ) }   +  T   \frac{c^{k+\nu} k^\nu }{ (\log T)^\nu }  
\end{align*}
for $  \frac12 \leq \sigma \leq  \frac12 + \frac{4}{ \log x}  $.
\end{lem}

\begin{proof}
Define two sets
\begin{align*}
 S_1 &= \left\{ t \in [T, 2T ] : \sigma_{x,t} > \max \left( \sigma ,  \frac12 + \frac{4}{ \log x}  \right) \right\} , \\
  S_2 &= \left\{ t \in [T, 2T ] : \sigma_{x,t} =  \frac12 + \frac{4}{ \log x} > \sigma    \right\} .
  \end{align*}
  Since $\sigma_{x,t} \geq  \frac12 + \frac{4}{ \log x}$, we see that
  $$ \int_{\substack{  \sigma_{x,t} > \sigma \\ T  \leq t \leq 2T}} ( \sigma_{x,t} - \sigma)^\nu  X^{ \sigma_{x,t} - \sigma}  dt  = \int_{S_1 } ( \sigma_{x,t} - \sigma)^\nu  X^{ \sigma_{x,t} - \sigma}  dt + \int_{S_2 } ( \sigma_{x,t} - \sigma)^\nu  X^{ \sigma_{x,t} - \sigma}  dt . $$

For $ t \in S_1 $, by the definition of $\sigma_{x,t}$ and $ \sigma_{x, t}>   \frac12 + \frac{4}{ \log x} $,   there exists  a zero $  \beta+ i \gamma$ such that  $\sigma_{x,t} =   2  \beta - \frac12  $, $ \beta - \frac12 > \frac{2}{ \log x}$ and $ | t- \gamma | \leq \frac{ x^{ 3 ( \beta - \frac12)}}{ \log x } $. Thus, we have
\begin{align*}
  \int_{S_1}   ( \sigma_{x,t} - \sigma)^\nu  X^{ \sigma_{x,t} - \sigma}  dt   
 & \leq \sum_{ \substack{ \beta > \frac12( \sigma + \frac12 )    \\     T/2\leq \gamma \leq 3T  }}  \int_{ \gamma - \frac{ x^{ 3 ( \beta - \frac12)}}{ \log x } }^{ \gamma + \frac{ x^{ 3 ( \beta - \frac12)}}{ \log x } }   
  \bigg(   2  \beta - \frac12  - \sigma \bigg)^\nu  X^{  2  \beta - \frac12   - \sigma}  dt \\
& \leq   \frac{ 2^{1+\nu} x^{ \frac32 ( \sigma - \frac12)}}{ \log x }    \sum_{ \substack{  \beta > \frac12( \sigma + \frac12 )    \\   T/2\leq \gamma \leq 3T  }}    \bigg(   \beta - \frac12 \bigg( \sigma + \frac12 \bigg)\bigg)^\nu  (x^3 X^2)^{   \beta - \frac12( \sigma + \frac12 )    }  .
\end{align*}
By Lemma \ref{Ts Lemma 5.1} the above is
\begin{equation}\label{Ts Lemma 5.2 bound 1}
\ll \frac{ (c\nu)^\nu k }{   ( \log T)^{ \nu}     }     T^{1 - \kappa (\sigma   -  \frac12 ) } x^{ \frac32 ( \sigma - \frac12)}
\end{equation}  
for some $c>0$.

We see that $S_2 = \emptyset$ for $ \sigma \geq  \frac12 + \frac{4}{ \log x }$. If $ \frac12 \leq  \sigma \leq   \frac12 + \frac{4}{ \log x }$, then 
$$   \int_{S_2 } ( \sigma_{x,t} - \sigma)^\nu  X^{ \sigma_{x,t} - \sigma}  dt \leq    T \bigg(  \frac{4}{ \log x} \bigg)^\nu  X^{ \frac{4}{ \log x}}   \leq    T   \frac{c^{k+\nu} k^\nu }{ (\log T)^\nu }  
$$
for some $c>0$.

\end{proof}

Next we consider \cite[Lemma 5.3]{Ts} and observe that the  condition (ii) therein does not hold in our setting. To adapt its proof to our setting, it requires several inequalities regarding $ \beta_L$. By assumptions A1 and A6 we have 
 \begin{equation}\label{eqn Ljpk}
 \beta_{L} (p^r ) =  \frac1r \sum_{i=1}^d \alpha_{i}(p)^r  .
 \end{equation}
From \eqref{eqn Ljpk} and assumption A1 it is easy to derive that
 \begin{equation}\label{beta bound 0}
  |\beta_{L}(p^r)| \leq \frac{d}{r}   p^{r\eta}  \quad \mathrm{for~} r \geq 1,
  \end{equation}
\be\label{beta bound 1} 
   | \beta_{L } (p^r) | \leq    \frac1r  \sum_{i=1}^d  | \alpha_{i}(p)|^r \leq \frac{p^{(r-2)\eta}}{r} \sum_{i=1}^d  |\alpha_{i} (p)|^2  \quad \mathrm{for~} r \geq 2 
 \ee 
 and  
\be\label{beta bound 2}
| \beta_{L} (p ) |^2  \leq    \bigg(  \sum_{i=1}^d  | \alpha_{i}(p)| \bigg)^2    \leq d  \sum_{i=1}^d  |\alpha_{i} (p)|^2     .
 \ee
 For  convenience we extend $ \beta_L$ by letting $ \beta_L (n) = 0 $ if $ n $ is not a power of a prime. Then we see that 
 $$ \log L(s) = \sum_n  \frac{ \beta_L (n)}{ n^s } . $$
  Define
  $$ \lambda_t := \lambda( \sigma, x, t) := \max\{ \sigma_{x,t}, \sigma \}   $$
for  $ \sigma \in [ \frac12, 1]$, then we have a following lemma. 
 \begin{lem}\label{Ts Lemma 5.3}
 Let $k $ and $m$ be positive integers such that $    k \leq m \leq 16 k $.  Let $ x = T^{ \frac{\varepsilon}{k}} $ and assume that  $ \frac{ \varepsilon}{k} <   \frac{\kappa}{3} $ and $0 < \varepsilon \leq \frac1{48}$.  Then  there exists a constant $c>0$ such that
$$    \int_T^{2T}  \bigg| \sum_{ n } \frac{ \beta_L (n) g_x(n) }{ n^{ \lambda_t +it}} \bigg|^{2m} dt    \ll T   c^k  k^m       \bigg(  \min\{ \log \log x ,  \log \frac{1}{  \sigma - \frac12  }           \} \bigg)^m    $$
and
$$    \int_T^{2T}  \bigg| \sum_{ n } \frac{ \beta_L (n) g_x(n)    \log n  }{ n^{ \lambda_t +it}} \bigg|^{2m} dt    \ll  T   c^k  k^m       \bigg(  \min  \{ \log x,  \frac{1}{   \sigma - \frac12}     \}  \bigg)^{2     m }    $$
 for $ \frac12 \leq \sigma \leq 1 $.
 \end{lem}

\begin{proof}
For a nonnegative integer $\ell  $ we see that 
$$\sum_{ n } \frac{ \beta_L (n) g_x(n) ( \log n)^\ell }{ n^{ \lambda_t +it}} = \sum_{ n } \frac{ \beta_L (n) g_x(n) ( \log n)^\ell}{ n^{ \sigma  +it}} + \sum_{ n } \frac{ \beta_L (n) g_x(n) ( \log n)^\ell }{ n^{   it}} ( n^{- \lambda_t} - n^{- \sigma}) . $$
We split the first sum on the right-hand side as
\begin{align*}
  \sum_{ n } \frac{ \beta_L (n) g_x(n)( \log n)^\ell }{ n^{ \sigma +it}} = & \sum_{ p } \frac{ \beta_L (p) g_x(p) ( \log p)^\ell}{ p^{ \sigma +it}}+\sum_{ p } \frac{ \beta_L (p^2 ) g_x(p^2 ) ( 2 \log p)^\ell}{ p^{ 2\sigma +2it}} \\
  &+\sum_{p}\sum_{ r \geq 3 } \frac{ \beta_L (p^r ) g_x(p^r ) ( r\log p)^\ell}{ p^{ r\sigma +irt}}   .
  \end{align*}
By \eqref{beta bound 1} and assumption A4 we have 
\begin{align*}
 \bigg|  \sum_{p}\sum_{ r \geq 3 } \frac{ \beta_L (p^r ) g_x(p^r ) (r \log p)^\ell  }{ p^{ r\sigma +irt}}   \bigg|  &  \leq \sum_{p}\sum_{ 3 \leq r \leq \frac{ 3 \log x}{ \log p}  } \frac{ p^{(r-2)\eta }\sum_{i=1}^d |\alpha_i (p  )|^2  ( r \log p)^\ell }{ r p^{ r\sigma  }} \\
   & \ll  \sum_{p}  \frac{ \sum_{i=1}^d |\alpha_i (p  )|^2  ( \log p)^\ell }{   p^{ \frac32  - \eta }} \ll 1 . 
\end{align*}
By \cite[Lemma 3.3]{Ts} we have
\begin{align*}
  \int_T^{2T} \bigg|    \sum_{ p } \frac{ \beta_L (p) g_x(p) ( \log p)^\ell}{ p^{ \sigma +it}} \bigg|^{2m} dt & \ll T m!   \bigg(  \sum_{ p } \frac{ |\beta_L (p) g_x(p)|^2 ( \log p)^{2\ell} }{ p^{ 2 \sigma  }} \bigg)^m  \\
    \int_T^{2T} \bigg|    \sum_{ p } \frac{ \beta_L (p^2) g_x(p^2)( \log p)^\ell }{ p^{ 2\sigma +2 it}} \bigg|^{2m} dt  & \ll T m!   \bigg(  \sum_{ p } \frac{ |\beta_L (p^2) g_x(p^2)|^2 ( \log p)^{2\ell}}{ p^{ 4 \sigma  }} \bigg)^m 
  \end{align*}
provided that $x^{3m } \ll T$, which holds for $ 0< \varepsilon \leq \frac1{48}$. By assumption A6 we have
 $$  \sum_{ p } \frac{ |\beta_L (p) g_x(p)|^2 ( \log p)^{2\ell}  }{ p^{ 2 \sigma  }} \leq \sum_{ p \leq x^3  } \frac{ |\beta_L (p)  |^2 ( \log p)^{2\ell} }{ p } \ll 
 \begin{cases}  
 \log \log x  & \mathrm{if}~ \ell=0 ,\\
 (\log x)^{2\ell }  & \mathrm{if}~ \ell \geq 1 
 \end{cases}   $$
 for $ \frac12 \leq \sigma \leq \frac12 + \frac{4}{ \log x} $, 
 \begin{align*}
   \sum_{ p } \frac{ |\beta_L (p) g_x(p)|^2 ( \log p)^{2\ell}  }{ p^{ 2 \sigma  }}& \leq \sum_{ p    } \frac{ |\beta_L (p)  |^2 ( \log p)^{2\ell}  }{ p^{2 \sigma}  } \ll \int_2^\infty    u^{-2\sigma}(\log u )^{2\ell-1}  du  \\
   &   \ll   \begin{cases}
    \log \frac{1}{  \sigma - \frac12  }  & \mathrm{if}~ \ell=0,\\
    \frac{1}{ (   \sigma - \frac12 )^{2 \ell }} & \mathrm{if}~ \ell \geq 1 
    \end{cases}
   \end{align*} 
 for $ \frac12 + \frac{4}{ \log x}  \leq \sigma \leq 1 $. By \eqref{beta bound 1} and assumption A4 we have
\begin{align*}
\sum_{ p } \frac{ |\beta_L (p^2) g_x(p^2)|^2 (\log p)^{2\ell}}{ p^{ 4 \sigma  }} \ll  \sum_{ p } \frac{ \sum_{i=1}^d |\alpha_i (p )   |^2 (\log p)^{2\ell}}{ p^{ 2-2 \eta  }} \ll 1
\end{align*}
for $ \sigma \geq \frac12 $. Since
\begin{multline*}
   \bigg| \sum_{ n } \frac{ \beta_L (n) g_x(n) (\log n)^\ell }{ n^{ \sigma +it}} \bigg|^{2m}  \\ 
  \leq   3^m \bigg(     \bigg|\sum_{ p } \frac{ \beta_L (p) g_x(p) (\log p)^{\ell} }{ p^{ \sigma +it}}  \bigg|^{2m}  +   \bigg| \sum_{ p } \frac{ \beta_L (p^2 ) g_x(p^2 ) (2 \log p  )^{\ell}}{ p^{ 2\sigma +2it}}\bigg|^{2m}  + c^m   \bigg) 
\end{multline*}
for some $c>0$, by collecting above equations we find that
\begin{equation}\label{eqn high moment beta sum} \begin{split}
\int_T^{2T}&  \bigg| \sum_{ n } \frac{ \beta_L (n) g_x(n) (\log n)^\ell }{ n^{ \sigma +it}} \bigg|^{2m} dt 
\\
 &  \ll  \begin{cases}
 T   c^k  k^m       \bigg(  \min\{ \log \log x ,  \log \frac{1}{  \sigma - \frac12  }           \} \bigg)^m       & \mathrm{if}~     \ell=0 ,\\
 T   c^k  k^m       \bigg(  \min  \{ \log x,  \frac{1}{   \sigma - \frac12}     \}  \bigg)^{2   \ell  m }      & \mathrm{if}~     \ell \geq   1 
 \end{cases}
\end{split}\end{equation}
for some constant $c>0$ and for $ \frac12 \leq \sigma \leq 1 $.

We next  estimate
$$ \int_T^{2T} \bigg|  \sum_{ n } \frac{ \beta_L (n) g_x(n) ( \log n)^\ell }{ n^{   it}} ( n^{- \lambda_t} - n^{- \sigma}) \bigg|^{2m} dt  . $$
By equations in \cite[p. 67]{Ts} the above integral is bounded by
\begin{align*}
\ll &\bigg( \int_T^{2T} ( \lambda_t - \sigma)^{4m} X_1^{4m(\lambda_t - \sigma)} dt \bigg)^{\frac12} \bigg( \int_{\sigma}^\infty X_1^{\sigma-v} dv \bigg)^{2m-\frac12}\\
& \times \bigg(  \int_\sigma^\infty  X_1^{\sigma-v} \int_T^{2T} \bigg| \sum_{n} \frac{ \beta_L(n) g_x(n) (\log n )^{\ell+1} \log ( X_1  n )}{ n^{ v+it}} \bigg|^{4m} dt dv \bigg)^{\frac12}
\end{align*}
with $ X_1= T^{\frac{\varepsilon_1 }{m}} $ for some $ \varepsilon_1 >0$.
Let $ \nu= 4m $ and $ X = X_1^{4m} = T^{4 \varepsilon_1 }  $ in Lemma \ref{Ts Lemma 5.2}.
One can easily check that the assumptions in Lemma \ref{Ts Lemma 5.2} follow from the assumptions in Lemma \ref{Ts Lemma 5.3}. Thus, there exists $c>0$ such that
$$ \int_T^{2T}  ( \lambda_t - \sigma)^{4m} X_1^{4m(\lambda_t - \sigma)} dt  = 
\int_T^{2T} ( \lambda_t - \sigma)^{4m} X_1^{4m(\lambda_t - \sigma)} dt    \ll        c^k  k^{4m}        T^{1 - \frac{ \kappa}{2} (\sigma   -  \frac12 ) }  ( \log T)^{- 4m}    $$ 
 for $ \frac12    \leq \sigma \leq 1 $.  
 By  \eqref{eqn high moment beta sum} we have
 \begin{align*}
\int_\sigma^\infty  &  X_1^{\sigma-v} \int_T^{2T} \bigg| \sum_{n} \frac{ \beta_L(n) g_x(n)(\log n)^{\ell+1} \log ( X_1  n )}{ n^{ v+it}} \bigg|^{4m} dt dv \\
& \ll   T   c^k k^{2m}      \bigg( \frac{ \log T}{k}\bigg)^{2m(2\ell+3) -1 }   \bigg(\min \{ \log x , \frac{1}{ \sigma - \frac12}  \} \bigg)^{2m}      .
\end{align*}
Therefore, by combining above results we obtain 
\begin{multline}\label{eqn high moment beta sum 2}
   \int_T^{2T} \bigg|  \sum_{ n } \frac{ \beta_L (n) g_x(n) (\log n)^\ell }{ n^{   it}} ( n^{- \lambda_t} - n^{- \sigma}) \bigg|^{2m} dt  \\
    \ll    c^k   k^{ 2m -2m\ell   } T^{ 1 - \frac{ \kappa}{4} ( \sigma - \frac12 )}    ( \log T)^{   2m\ell - m }    \bigg(\min \{ \log x , \frac{1}{ \sigma - \frac12}  \} \bigg)^{m}   
    \end{multline}
for $ \frac12   \leq \sigma \leq 1 $. The lemma follows from \eqref{eqn high moment beta sum} and \eqref{eqn high moment beta sum 2}.
\end{proof}

The following lemma is an analogy of \cite[Lemma 5.4]{Ts}. 
 The proof of \cite[Lemma 8]{Le1} is for Hecke $L$-functions of number fields, but it works also for our $L$-functions. So we state the lemma without a proof.
\begin{lem}\label{Ts Lemma 5.4}
Let $ t \in [T, 2T] $, $ \frac12 \leq \sigma \leq 1 $ and $t \neq \im(\rho)$ for any zeros $\rho$ of $L(s)$. Then we have
\begin{multline*}
  \log L(s) = \sum_n \frac{ \beta_L (n) g_x(n)}{ n^{ \lambda_t + it}} + \tilde{L} (s) \\
  + O \bigg(   \bigg(  \frac{  x^{ \frac14 - \frac12 \lambda_t}}{ \log x} + ( \lambda_t - \sigma) \bigg)\bigg(  \bigg| \sum_n \frac{ \beta_L (n) g_x(n) \log n }{ n^{ \sigma_{x,t} + it}} \bigg| + \log T       \bigg)\bigg),
\end{multline*}
where
$$ \tilde{L}  ( s) = \sum_{\rho}  \int_\sigma^{\lambda_t}    \frac{    u- \lambda_t}{  (u+it-\rho)(\lambda_t + it - \rho)} du . $$
\end{lem}
The following lemma can be derived from the same arguments as in the proof of \cite[Lemma 5.5]{Ts}, so we state it without a proof. 
\begin{lem}\label{Ts Lemma 5.5}
\begin{align*}
|\im( \tilde{L} (s)) |   & \ll ( \lambda_t - \sigma) \bigg(   \bigg|  \sum_n \frac{ \beta_L (n) g_x(n) \log n }{ n^{ \lambda_t  + it}}\bigg| + \log T \bigg),  \\
|\re(\tilde{L}(s))| & \ll  ( \lambda_t - \sigma)\big(  1+ ( \lambda_t - \sigma) \log x + \log^+  \frac{1}{  \eta_t \log x }   \big) \bigg(   \bigg|  \sum_n \frac{ \beta_L (n) g_x(n) \log n }{ n^{ \lambda_t + it}}\bigg| + \log T \bigg),
\end{align*}
where $ \log^+ w := \max\{ \log w , 0 \} $ and $ \eta_t =  \min | t-\gamma|$ with the minimum taken over all zeros $ \beta+i \gamma $ of $L(s)$ with $ \beta \geq \frac12$. Moreover, we have
$$  \int_T^{2T} \left( \log^+  \frac{1}{  \eta_t \log x }  \right)^{2k} dt \ll T(ck)^{2k} . $$
\end{lem}

\subsection{Proof of Theorems \ref{thm 2nd moment} and \ref{thm high moment}} \label{sec proof moment}

First we want to find an upper bound of the $2k$-th moment
$$ \int_T^{2T} \bigg| \log L ( \sigma_T + it ) - \sum_{ n} \frac{ \beta_L (n) g_x (n) }{ n^{ \sigma_T + it }} \bigg|^{2k} dt. $$ 
Let $\sigma = \frac12$, $\ell=1$ and $ k=m$ in Lemma \ref{Ts Lemma 5.3}, then we get
\begin{equation}\label{proof thm2 eqn1}
\int_T^{2T}  \bigg| \sum_n \frac{ \beta_L (n) g_x(n) \log n }{ n^{ \sigma_{x,t} + it}} \bigg|^{2k} dt \ll   c^k    k^k  T ( \log x)^{2  k}. 
\end{equation}
By  Lemmas \ref{Ts Lemma 5.4} and \ref{Ts Lemma 5.5} and \eqref{proof thm2 eqn1},  we have
\begin{equation}\label{proof thm2 eqn2}\begin{split}
\int_T^{2T}&  \bigg| \log L ( \sigma_T + it ) - \sum_{ n} \frac{ \beta_L (n) g_x (n) }{ n^{ \sigma_T + it }} \bigg|^{2k} dt \\
\ll & c^k \int_T^{2T} \bigg|   \sum_n \frac{ \beta_L (n) g_x(n)}{ n^{ \lambda_t + it}} -  \sum_n \frac{ \beta_L (n) g_x(n)}{ n^{ \sigma_T + it}} \bigg|^{2k} dt     \\
+& c^k \int_T^{2T}  ( \lambda_t - \sigma_T )^{2k}\bigg(  1+ ( \lambda_t - \sigma_T ) \log x + \log^+  \frac{1}{  \eta_t \log x }   \bigg)^{2k}  \bigg|  \sum_n \frac{ \beta_L (n) g_x(n) \log n }{ n^{ \lambda_t + it}}\bigg|^{2k}  dt  \\
+& c^k ( \log T)^{2k}  \int_T^{2T}  ( \lambda_t - \sigma_T )^{2k}\big(  1+ ( \lambda_t - \sigma_T ) \log x + \log^+  \frac{1}{  \eta_t \log x }   \big)^{2k}    dt  \\
+&  c^k   k^{2k} T e^{- \varepsilon \frac{ \log T}{ G(T)} } 
\end{split}\end{equation}
for some $c>0$. It remains to bound the integrals on the right-hand side.

Since $ x= T^{ \frac{ \varepsilon}{k}}$  and $ k \leq \frac{\varepsilon}{4}  ( \log \log T)^2 $, we see that
$$ \sigma_T - \frac12 =  \frac{1}{G(T)} \geq   \frac{ ( \log \log T)^2 }{ \log T}  \geq    \frac{4}{ \log x }. $$
By \eqref{eqn high moment beta sum 2} we have
\begin{equation}\label{proof thm2 eqn3}
 \int_T^{2T} \bigg|   \sum_n \frac{ \beta_L (n) g_x(n)}{ n^{ \lambda_t + it}} -  \sum_n \frac{ \beta_L (n) g_x(n)}{ n^{ \sigma_T + it}} \bigg|^{2k} dt  \ll  c^k   k^{ 2k    } T e^{ - \frac{ \kappa}{4} \frac{ \log T}{ G(T) } }  \frac{  G(T)^k   }{ ( \log T)^k }  
 \end{equation}
for some $ c>0$.
By Lemmas \ref{Ts Lemma 5.2} and \ref{Ts Lemma 5.5} we have
$$ \int_T^{2T} ( \lambda_t - \sigma)^{2m} dt  \ll   \frac{ c^k m^{2m}}{ ( \log T)^{2m}}  T e^{  - ( \kappa - \frac{ 3\varepsilon}{ 2k } ) \frac{ \log T}{G(T)} }    $$  
and  
 \begin{align*}
\int_T^{2T} & ( \lambda_t - \sigma)^{2m}      \bigg( \log^+  \frac{1}{  \eta_t \log x }\bigg)^{2m}   dt    \\
& \leq \bigg( \int_T^{2T} ( \lambda_t - \sigma)^{4m} dt \bigg)^{\frac12}    \bigg( \int_T^{2T}      \bigg( \log^+  \frac{1}{  \eta_t \log x }\bigg)^{4m}dt \bigg)^{\frac12}   \\ 
& \ll        \frac{ c^k m^{4m}}{ ( \log T)^{2m}}  T e^{  - ( \frac{\kappa }{2} - \frac{ 3\varepsilon}{ 4k } ) \frac{ \log T}{G(T)} }    
  \end{align*}
  for $ k \leq m \leq 4k $.  Thus, we obtain
\begin{equation}\label{proof thm2 eqn4}\begin{split}
 \int_T^{2T}  &  ( \lambda_t - \sigma_T )^{2m}\big(  1+ ( \lambda_t - \sigma_T ) \log x + \log^+  \frac{1}{  \eta_t \log x }   \big)^{2m}       dt   \\
 &   \ll    \frac{ c^k m^{4m}}{ ( \log T)^{2m}}  T e^{  - ( \frac{\kappa }{2} - \frac{ 3\varepsilon}{ 4k } ) \frac{ \log T}{G(T)} } 
 \end{split}\end{equation}
 for $ k \leq m \leq  2k $.
By  Lemma \ref{Ts Lemma 5.3}, the Cauchy-Schwarz inequality and the above inequality we have
\begin{equation}\label{proof thm2 eqn5} \begin{split}
  \int_T^{2T}  & ( \lambda_t - \sigma_T )^{2k}\bigg(  1+ ( \lambda_t - \sigma_T ) \log x + \log^+  \frac{1}{  \eta_t \log x }   \bigg)^{2k}  \bigg|  \sum_n \frac{ \beta_L (n) g_x(n) \log n }{ n^{ \lambda_t + it}}\bigg|^{2k}  dt  \\
 &  \ll   \frac{ c^k k^{5k}    G(T)^{2k}  }{ ( \log T)^{2k}}  T e^{  - ( \frac{\kappa }{4} - \frac{ 3\varepsilon}{ 8k } ) \frac{ \log T}{G(T)} } .
 \end{split} \end{equation}
   Therefore, by \eqref{proof thm2 eqn2} -- \eqref{proof thm2 eqn5} there exist $ \kappa_0 >0 $ such that    
\begin{equation}\label{proof thm2 eqn6}
\int_T^{2T}   \bigg| \log L ( \sigma_T + it ) - \sum_{ n} \frac{ \beta_L (n) g_x (n) }{ n^{ \sigma_T + it }} \bigg|^{2k} dt \ll     c^k k^{4k}   T e^{  - \kappa_0  \frac{ \log T}{G(T)} }.
\end{equation}

Let $ k=1$ in \eqref{proof thm2 eqn6}, then we see that 
\begin{equation} \label{proof thm2 eqn11}
\int_T^{2T}   \bigg| \log L ( \sigma_T + it ) - \sum_{ n} \frac{ \beta_L (n) g_x (n) }{ n^{ \sigma_T + it }} \bigg|^{2} dt \ll     T e^{  - \kappa_0  \frac{ \log T}{G(T)} }.
\end{equation}
Here, $ x = T^\varepsilon$ and $  0< \varepsilon < \min\{  \frac{1}{48},    \frac{\kappa}3   \} $.  Let $ e^{ \frac{G(T)}{2}} \leq Y \leq x $, then  we have
\begin{equation} \label{proof thm2 eqn12}
 \int_T^{2T}   \bigg|   \sum_{ n>Y  } \frac{ \beta_L (n) g_x (n) }{ n^{ \sigma_T + it }} \bigg|^{2} dt \ll T   \sum_{   n >Y  } \frac{ | \beta_L (n)   |^2  }{ n^{2 \sigma_T   }} \ll T \frac{Y^{1-2\sigma_T}  }{  (2 \sigma_T - 1 ) \log Y } 
 \end{equation}
by \cite[Lemma 4.1]{LL}. Thus, Theorem \ref{thm 2nd moment} follows from \eqref{proof thm2 eqn11} and \eqref{proof thm2 eqn12}.

 Next we prove Theorem \ref{thm high moment}.  
We see that \eqref{thm 2 eqn 1} holds by \eqref{eqn high moment beta sum} and \eqref{proof thm2 eqn6}. 
The proof of \eqref{thm 2 eqn 2} is similar, but simpler than the proof of Lemma \ref{Ts Lemma 5.3}. Since
\begin{align*}
  \log L( \sigma_T , \X ) = & \sum_{ p } \frac{ \beta_L (p) \X (p) }{ p^{ \sigma_T }}+\sum_{ p } \frac{ \beta_L (p^2 ) \X (p^2 ) }{ p^{ 2\sigma_T } } +O(1) , 
  \end{align*}
  by  \cite[Lemma 3.3]{Ts} we have
\begin{align*}
\E[  | \log L( \sigma_T , \X )|^{2k} ]   \leq  & c^k \bigg(    k!  \bigg( \sum_{ p } \frac{ | \beta_L (p) |^2  }{ p^{ 2\sigma_T }} \bigg)^k + k! \bigg( \sum_{ p } \frac{ | \beta_L (p^2 )|^2  }{ p^{ 4\sigma_T } } \bigg)^k  + 1 \bigg)  
  \end{align*}
  for some $ c > 0 $.
  By \eqref{beta bound 1} and assumption A4 we have
\begin{align*}
\sum_{ p } \frac{ |\beta_L (p^2)  |^2  }{ p^{ 4 \sigma_T   }} \ll  \sum_{ p } \frac{ \sum_{i=1}^d |\alpha_i (p )   |^2 }{ p^{ 2-2 \eta  }} \ll 1
\end{align*}
     By assumption A6 we have
 $$   \sum_{ p   } \frac{ |\beta_L (p)  |^2 }{ p^{2 \sigma_T} } \ll  \int_2^\infty  \frac{ du}{  u^{1+ \frac{2}{ G(T)}} \log u } \ll  \log G(T).$$
Thus, we have
$$
\E[  | \log L( \sigma_T , \X )|^{2k} ]  \ll c^k k! ( \log G(T) )^k  $$
for some $ c >0$.

\section{Discrepancy}\label{sec disc}

In this section we will prove Theorem \ref{thm disc} for $G(T)$ satisfying \eqref{GT range}.  First we need to extend \cite[Proposition 5.1]{LL}.
Define the Fourier transforms of $\Phi_T $ and $ \Phrand $  by
$$ \Phhat ( \xb, \yb ) := \int_{\mathbb{R}^{2J}} e^{  2\pi i ( \xb \cdot \ub + \yb \cdot \vb)} d\Phi_T ( \ub, \vb) $$
and 
$$ \Phrhat ( \xb, \yb ) := \int_{\mathbb{R}^{2J}} e^{  2\pi i ( \xb \cdot \ub + \yb \cdot \vb)} d\Phrand ( \ub, \vb) ,$$
where  $ \xb = ( x_1 , \ldots , x_J ) $ and similarly $ \yb, \ub, \vb$ are vectors in $ \mathbb{R}^J$ and  $ \xb \cdot \ub  := \sum_{ j \leq J } x_j u_j $ is the dot product. Then we obtain the following proposition.
\begin{pro}\label{pro disc trans}
Assume   \eqref{GT range}. Given constant $A_4>0$, there exists a constant $ A_5 >0$ such that 
 $$ \Phhat( \xb, \yb) = \Phrhat (\xb, \yb) +   O \bigg(     \frac{1}{ ( \log T)^{A_4}} \bigg) $$
  for $\max_{ j \leq J} \{ |x_j | , |y_j | \}  \leq  \frac{  \sqrt{ \log T}}{ A_5 \sqrt{ G(T)} \log \log T} $.
\end{pro}

\begin{proof}
 By definition we get
 \begin{align*}
 \Phhat( \xb, \yb ) &  = \frac1T \int_T^{2T} \exp \bigg[      2 \pi i  \sum_{ j \leq J} \big( x_j \log | L_j ( \sigma_T  + it ) | + y_j \arg L_j ( \sigma_T + i t )   \big) \bigg] dt, \\ 
 \Phrhat( \xb, \yb) & = \mathbb{E}  \bigg[   \exp  \bigg[     2 \pi i  \sum_{ j \leq J} \big( x_j \log | L_j ( \sigma_T , X ) | + y_j \arg L_j ( \sigma_T , \X )   \big) \bigg]   \bigg].
 \end{align*}
Since the inequality
$$  | e^{ix} - e^{iy} |^2  = 4 \sin^2 \bigg( \frac{ x-y}{2}  \bigg)   \leq |x-y|^2  $$
holds for  any $ x,y \in \mathbb{R} $, by the Cauchy-Schwarz inequality and Theorem \ref{thm 2nd moment} with 
$$ \log Y = A_6 G(T) \log \log T    $$
we have
\begin{align*}
 \Phhat( \xb, \yb ) &   -  \frac1T \int_T^{2T} \exp \bigg[      2 \pi i  \sum_{ j \leq J} \big( x_j \re ( R_{j ,Y} ( \sigma_T  + it ))  + y_j \im ( R_{j,Y}   ( \sigma_T + i t ) )   \big) \bigg] dt \\
  = & O \bigg(   \frac1T \int_T^{2T}    \sum_{j \leq J} ( |x_j| + |y_j | ) |  \log L_j ( \sigma_T + it ) - R_{j,Y} ( \sigma_T + it ) |    dt \bigg) \\
  = & O \bigg(    \sum_{j \leq J} ( |x_j| + |y_j | )  \bigg( \frac1T \int_T^{2T} |  \log L_j ( \sigma_T + it ) - R_{j,Y} ( \sigma_T + it ) |^2     dt \bigg)^{ \frac12}  \bigg) \\
  = & O \bigg(  \frac{ M }{ ( \log T)^{A_6}}    \bigg)  
\end{align*}
for all $ |x_j | , |y_j | \leq M$.   
 Let 
 $$ N = \bigg[  \frac{ \log T}{ 10 A_6 G(T) \log \log T} \bigg] ,$$
then by the Taylor theorem and \cite[Lemma 4.5]{LL}
  we have
\begin{align*}
 \Phhat( \xb, \yb ) - &   \sum_{n=0}^{2N-1} \frac{  (2 \pi i)^n  }{ n! T }     \int_T^{2T} \bigg(        \sum_{ j \leq J} \big( x_j \re ( R_{j ,Y} ( \sigma_T  + it ))  + y_j \im ( R_{j,Y}   ( \sigma_T + i t ) )   \big) \bigg)^n  dt \\
=  &   O \bigg(  \frac{ c^N M^{2N} }{(2N)! }        \frac1T \int_T^{2T}         \sum_{ j \leq J} \big| R_{j,Y}   ( \sigma_T + i t )     \big|^{2N}  dt    +  \frac{ M }{( \log T)^{A_6}}      \bigg)\\
  =  &   O \bigg(      \bigg(  \frac{ c M^2   \log \log T}{N} \bigg)^N     +  \frac{ M }{( \log T)^{A_6}}   \bigg)  
\end{align*}
  for some $ c>0$. 
Let 
$$ M  =  \frac{  \sqrt{ \log T}}{ A_5 \sqrt{ G(T)} \log \log T} $$
with a constant $A_5  \geq \sqrt{ 10 c A_6   }e^{ 5A_6^2} $, then we have 
\begin{align*}
 \Phhat( \xb, \yb ) = &   \sum_{n=0}^{2N-1} \frac{  (2 \pi i)^n  }{ n! T }     \int_T^{2T} \bigg(        \sum_{ j \leq J} \big( x_j \re ( R_{j ,Y} ( \sigma_T  + it ))  + y_j \im ( R_{j,Y}   ( \sigma_T + i t ) )   \big) \bigg)^n  dt \\
   &  +   O \bigg(     \frac{ 1 }{( \log T)^{A_6 - \frac12}}   \bigg) .  
\end{align*}
By following the second half of the proof of \cite[Proposition 5.1]{LL} one can conclude that the proposition holds.

\end{proof}

We next need to introduce Beurling-Selberg functions.   
Define 
$$ F_{ [a,b], \Delta} (z) = \frac12  (   H( \Delta(z-a) ) - K ( \Delta ( z-a) )  +  H( \Delta ( b-z)) - K( \Delta (b-z)) )   $$
for $ z\in \mathbb{C}$ and $ \Delta > 0$, where
$$ H(z) = \frac{   \sin^2 (\pi z) }{ \pi^2 } \bigg(   \sum_{n=-\infty}^\infty \frac{ \mathrm{sgn} (n) }{(z-n)^2 } + \frac{2}{z} \bigg) \quad \mathrm{and} \quad   K(z) = \frac{  \sin^2 (\pi z) }{ ( \pi z)^2}. $$
Then we summarize some results in \cite[Section 7]{LLR1} as a lemma.
\begin{lem}\label{lem BS}
For all $x \in \mathbb{R}$ we have $| F_{ [a,b], \Delta} (x) | \leq 1 $ and 
$$ 0 \leq \mathbf{1}_{[a,b] } (x)  - F_{ [a,b], \Delta} (x)  \leq K (  \Delta ( x-a)) + K( \Delta ( b-x )) . $$
Moreover, the Fourier transform $ \widehat{F}_{ [a,b], \Delta} $ satisfies
$$  \widehat{F}_{ [a,b], \Delta}  =  \begin{cases}
\widehat{\mathbf{1}}_{[a,b]  }(y) +  O ( \Delta^{-1})  & \mathrm{if} ~ |y| \leq \Delta , \\
0 & \mathrm{if}~ |y| \geq \Delta.
\end{cases}$$
\end{lem}

 We  are ready to prove Theorem \ref{thm disc} for $G(T)$ satisfying \eqref{GT range}. By Corollary \ref{cor large dev} there exists a constant $A_3>0$ such that 
\begin{align*}
 \frac{1}{T} \me \{ t \in [T, 2T] : \Lv ( \sigma_T + it ) \notin  I_T  \} & \ll \frac{1}{ ( \log T)^{10}} ,  \\
 \P \{  \Lv ( \sigma_T , \X  ) \notin  I_T \} & \ll \frac{1}{ ( \log T)^{10}},
\end{align*}
where
 $$ I_T := [- A_3 \log \log T , A_3 \log \log T ]^{2J} .$$
 Then we see that 
 \begin{align*}
 \Phi_T(  \R) & = \Phi_T ( \R \cap I_T  ) + O \bigg(  \frac{1}{ ( \log T)^{10}} \bigg), \\
 \Phrand (  \R) & = \Phrand ( \R \cap I_T   ) + O \bigg(  \frac{1}{ ( \log T)^{10}} \bigg)
  \end{align*}
for any $\R \in \mathbb{R}^{2J}$. Thus, we have
\begin{equation}\label{D sigma T eqn 1}
  \D(\sigma_T) =  \sup_{\R  \subset I_T  } | \Phi_T ( \R) -  \Phrand ( \R) | + O \bigg(  \frac{1}{ ( \log T)^{10}} \bigg) , 
  \end{equation}
 where $\R  \subset I_T  $ runs over all rectangular boxes of $ \mathbb{R}^{2J}$ with  sides parallel to the coordinate axes.
   By \eqref{D sigma T eqn 1} it is enough to show that
\begin{equation}\label{PhiT Phrand eqn 1}
\Phi_T(\R) - \Phrand(\R) = O( M^{-1})    
\end{equation}
for
$$ \R = \prod_{j=1}^J  I_{1,j} \times \prod_{j=1}^J  I_{2,j} \subset I_T  ,$$
where $I_{1,j} = [a_j, b_j ]$ and $ I_{2,j} = [c_j, d_j ] $ for $ j = 1 , \ldots , J$.

 By definition we see that
 \begin{align*}
  \Phi_T ( \R) & =   \frac{1}{T} \int_T^{2T} \prod_{j=1}^J  \mathbf{1}_{ I_{1,j} } (   \log | L_j ( \sigma_T + i t )|   ) \mathbf{1}_{ I_{2,j}  } ( \arg L_j ( \sigma_T + i t )  ) dt    , \\
   \Phrand (\R) &  = \E \bigg[ \prod_{j=1}^J  \mathbf{1}_{ I_{1,j} } (   \log | L_j ( \sigma_T , \X  )|   )  \mathbf{1}_{ I_{2,j} } ( \arg L_j ( \sigma_T, \X)  )    \bigg] .
   \end{align*}
 By Lemma \ref{lem BS} with $\Delta = M$  we have
 \begin{equation}\label{proof disc eqn1}\begin{split}
 \Phi_T ( \R)  & =   \frac{1}{T} \int_T^{2T} \prod_{j=1}^J  F_{ I_{1,j} , M} (   \log | L_j ( \sigma_T + i t )|   ) F_{ I_{2,j} , M } ( \arg L_j ( \sigma_T + i t )  ) dt + O( M^{-1} )  , \\ 
 \Phrand (\R) & = \E \bigg[ \prod_{j=1}^J  F_{ I_{1,j} , M} (   \log | L_j ( \sigma_T , \X  )|   ) F_{ I_{2,j} , M} ( \arg L_j ( \sigma_T, \X)  )    \bigg] + O( M^{-1} )  .
 \end{split}\end{equation}
 To confirm  the above $O$-terms, it requires inequalities similar to 
\begin{align*}
   \frac{1}{T} \int_T^{2T} & K( M (  \log |L_1 ( \sigma_T + it ) | - \alpha) ) dt   \\
   &  =  \frac{1}{M} \int_{-M}^M \bigg( 1 - \frac{ |u|}{M} \bigg) e^{- 2 \pi i \alpha u } \Phhat ( u , 0 , \ldots , 0 ) du  \ll  \frac{1}{M},
\end{align*} 
which holds  by Fourier inversion, Proposition \ref{pro disc trans},   \cite[Lemma 7.1]{LL} and
 $$ \hat{K} (x) = \max (  0 , 1- |x| ) . $$
 
 By Fourier inversion, Lemma \ref{lem BS} and Proposition \ref{pro disc trans} we obtain 
 \begin{equation}\label{proof disc eqn2}\begin{split}
&  \frac{1}{T} \int_T^{2T} \prod_{j=1}^J  F_{ I_{1,j} , M} (   \log | L_j ( \sigma_T + i t )|   ) F_{ I_{2,j} , M } ( \arg L_j ( \sigma_T + i t )  ) dt \\
& = \int_{ \mathbb{R}^{2J}}  \bigg(  \prod_{j=1}^J  \widehat{F}_{ I_{1,j} , M} (x_j) \widehat{F}_{ I_{2,j} , M} (y_j) \bigg)  \Phhat ( - \xb, - \yb ) d \xb d \yb \\ 
& = \int_{ \substack{ |x_j|, |y_j | \leq M  \\  j = 1, \ldots , J } }  \bigg(  \prod_{j=1}^J  \widehat{F}_{ I_{1,j} , M} (x_j) \widehat{F}_{ I_{2,j} , M} (y_j) \bigg)  \Phrhat ( - \xb, - \yb ) d \xb d \yb  + O\bigg(  \frac{(  M  \log \log T )^{2J}}{( \log T)^{A_4}}    \bigg) \\
&= \E \bigg[ \prod_{j=1}^J  F_{ I_{1,j} , M} (   \log | L_j ( \sigma_T , \X  )|   ) F_{ I_{2,j} , M} ( \arg L_j ( \sigma_T, \X)  )    \bigg] + O\bigg(  \frac{(  M  \log \log T )^{2J}}{( \log T)^{A_4}}    \bigg)  .
 \end{split}\end{equation}
Here, we also have used that 
$$|\hat{F}_{[a,b], M } (y)  |  \leq   |  \hat{\mathbf{1}}_{[a,b]} ( y) | + O(M^{-1}) \ll  \log \log T   $$
for $ |y| \leq M $ and $ | b-a | \ll \log \log T $. We choose  $A_4 $ sufficiently large so that
$$ \frac{(  M  \log \log T )^{2J}}{( \log T)^{A_4}}  \leq \frac{1}{M} , $$
then  \eqref{PhiT Phrand eqn 1} holds by \eqref{proof disc eqn1} and \eqref{proof disc eqn2}. This completes the proof of Theorem \ref{thm disc}.

\section*{Acknowledgements}

This work has been supported by the National Research Foundation of Korea (NRF) grant funded by the Korea government (MSIP) (No. 2019R1F1A1050795).

\end{document}